\documentclass[leqno,12pt]{amsart}
\usepackage{amsmath,amssymb,amsthm,amsfonts,amscd}
\usepackage[dvips]{graphicx}

\newcommand{\R}{\mathbf{R}}
\newcommand{\Z}{\mathbf{Z}}

\newcommand{\tc}{{\sf {TC}}}
\newcommand{\kk}{{\mathbf k}}
\newcommand{\C}{{\mathbf C}}

\newcommand{\x}{\mathtt{x}}
\newcommand{\co}{\colon\thinspace}

\newcommand{\TC}{{\sf {TC}}}

\newcommand{\gen}{\mathfrak{genus}}

\newcommand{\bi}{\begin{itemize}}
\newcommand{\ei}{\end{itemize}}
\newcommand{\vv}{{\mathfrak v}}

\newtheorem{thm}{Theorem}
\newtheorem{theorem}{Theorem}

\newtheorem*{thm5.2}{Theorem 5.2}

\begin{document}
\thanks{The research was supported by grants from the EPSRC and from the Royal Society.}
%\abbrevauthors{M. ~Farber, M. ~Grant}
%\abbrevtitle{}

\title[Topological complexity]{Topological complexity of configuration
spaces}

\author[M. ~Farber and M. ~Grant]{Michael Farber and Mark Grant}
\address{Department of Mathematical Sciences, Durham University\\
South Road, Durham, DH1 3LE, UK}

\email{michael.farber@durham.ac.uk}
\email{mark.grant@durham.ac.uk}

\maketitle
\begin{abstract}
The topological complexity $\tc(X)$ is a homotopy invariant which reflects the complexity of the problem of constructing a motion planning algorithm in the space $X$, viewed as configuration space of a mechanical system. In this paper we complete the computation of the topological complexity of the configuration space of
$n$ distinct points in Euclidean $m$-space for all $m\ge 2$ and $n\ge 2$; the answer was previously known in the cases $m=2$ and $m$ odd. We
also give several useful general results concerning sharpness of upper bounds for the
topological complexity.
\end{abstract}

\section{Introduction}
The motion planning problem is a central theme of robotics \cite{La}. Given a mechanical
system $S$, a motion planning algorithm for $S$ is a function associating with any
pair of states $(A,B)$ of $S$ a continuous motion of the system starting at $A$ and
ending at $B$. If $X$ denotes the configuration space of the system, one considers the
path fibration
\begin{equation}\label{fibration}
\pi: PX \to X\times X, \quad \pi(\gamma) = (\gamma(0), \gamma(1)),
\end{equation}
where $PX=X^I$ is the space of all continuous paths $\gamma: I=[0,1]\to X$.
In these terms, a motion planning algorithm for $S$ is a section (not necessarily
continuous) of $\pi$.

The topological complexity of a topological space $X$, denoted $\tc(X)$, is defined to
be the genus, in the sense of Schwarz \cite{Sch}, of fibration (\ref{fibration}).
More explicitly, $\TC(X)$ is the minimal integer $k$ such that $X\times X$ admits a cover
by $k$ open subsets, on each of which there exists a continuous
local section of fibration (\ref{fibration}). One of the basic properties of $\tc(X)$ is its homotopy invariance \cite{Fa03}.
If $X$ is a Euclidean Neighbourhood Retract then the number $\tc(X)$ can be equivalently characterized (see \cite{Finv}, Proposition 4.2)
as the minimal integer $k$ such that
there exists a section $s: X \times X\to PX$ of (\ref{fibration}) and a decomposition
$$X\times X= G_1\cup \dots\cup G_k, \quad G_i\cap G_j =\emptyset, \quad i\not=j$$
where each $G_i$ is locally compact and such that the restriction $s|G_i: G_i \to PX$ is continuous for $i=1, \dots,k$.
A section $s$ as above can be viewed as a motion planning algorithm: given a pair of states $(A, B)\in X\times X$ the path $s(A, B)(t)$ represents a continuous motion of the system starting from $A$ and ending at $B$. The number $\tc(X)$ is a measure of the complexity of motion planning algorithms for a system whose configuration space is $X$.

The concept $\tc(X)$ was
introduced and studied in \cite{Fa03}, \cite{Fa04}.  We refer the reader to surveys \cite{Fa06}, \cite{Finv} for detailed treatment of the invariant $\tc(X)$.
Computation of $\tc(X)$ in various practically interesting examples has received much
recent interest, see for instance papers \cite{CP1}, \cite{CP2}, \cite{FG07},
\cite{Gon}, \cite{G}.

In this paper we study the topological complexity $\tc(F(\R^m, n))$ of the space of
configurations of $n$ distinct points in Euclidean $m$-space. Here $m, n \ge 2$, and
$$F(\R^m,n)=\left\{ (\x_1,\ldots ,\x_n)\in (\R^m)^{\times n};
 \x_i\neq \x_j \, \mbox{for} \, i\neq j \right\},$$
topologised as a subspace of the Cartesian power $(\R^m)^{\times n}$.
This space appears in robotics when one controls multiple objects simultaneously trying to avoid collisions between them.
 Our main result in this paper is the following.
\begin{theorem}\label{thm1}
One has
\begin{equation}
\TC(F(\R^m,n))=\left\{\begin{array}{ll}
2n-1 & \textrm{for all }m\textrm{ odd}, \\
2n-2 & \textrm{for all }m\textrm{ even.}
\end{array}\right.
\end{equation}
\end{theorem}
The cases $m=2$ and $m\geq 3$ odd of Theorem \ref{thm1} were proven by Farber and
Yuzvinsky in \cite{FY04}, where it was conjectured that $\TC(F(\R^m,n))=2n-2$ for all even
$m$. Here we settle this conjecture in the affirmative. Note that the methods employed in \cite{FY04} are not applicable in the case when $m > 2$ is even. We therefore suggest an alternative approach based on sharp upper bounds for the topological complexity.

The plan of the paper is as follows. In the next section we state Theorems \ref{thm2} \and \ref{thm3} about sharp upper bounds; their proofs appear in section \S \ref{prfs}. The concluding section \S \ref{prf} contains the proof of Theorem \ref{thm1}.

\section{Sharp upper bounds for the topological complexity}\label{sharp}

Let $X$ be a CW-complex of finite dimension $\dim(X)=n\geq 1$.
We denote by $\Delta_X\subset X\times X$ the diagonal $\Delta_X=\left\{(\x, \x); \x \in X\right\}$.
Let $A$ be a
local system of coefficients on $X\times X$.
A cohomology class $$u\in H^*(X\times X;A)$$ is called {\it a zero-divisor} if
its restriction to the diagonal is trivial, i.e. $u|{\Delta_X}=0\in H^\ast(X; A|X)$. The importance of zero-divisors stems from the following fact (see \cite{Finv}, Corollary 4.40):

{\it If the cup-product of $k$ zero-divisors $u_i\in H^\ast(X\times X; A_i)$, where $i=1, \dots, k$, is nonzero then $\tc(X) >k$.}

Theorem \ref{thm2} below supplements the general dimensional upper bound of \cite{Fa03} by giving necessary and sufficient conditions for its sharpness.

\begin{theorem}\label{thm2}
For any $n$-dimensional cell complex $X$ one has
 \bi
\item[(a)] $\TC(X)\leq 2n+1$;
\item[(b)] $\TC(X)=2n+1$ if and only if there exists a local coefficient
 system $A$ on $X\times X$ and a zero-divisor $\xi\in
  H^1(X\times X;A)$ such that the $2n$-fold cup product
  $$\xi^{2n}=\xi\cup \dots\cup \xi \neq   0\in H^{2n}(X\times X; A^{2n})$$
  is nonzero. Here $A^{2n}$ denotes the tensor product of $2n$ copies
  $A\otimes \dots \otimes A$
  of $A$ (over $\Z$).
\ei
\end{theorem}
Next we state a similar sharp upper bound result for $(s-1)$-connected spaces $X$ where $s>1$.
We use the following notation.
 If $B$ is an abelian group and $v\in
H^r(X;B)$ is a cohomology class then the class
\begin{eqnarray}\label{bar}
\bar{v} =v\times 1 - 1\times v\in
H^r(X\times X;B)\end{eqnarray}
is a zero-divisor, where $1\in H^0(X;\Z)$ is the unit and $\times$ denotes the cohomological cross-product.

We say that
a finitely generated abelian group is {\it
  square-free} if it has no subgroups isomorphic to $\Z_{p^2}$, where
$p$ is a prime.

\begin{theorem}\label{thm3} Let $X$ be a $(s-1)$-connected $n$-dimensional finite cell complex where $s\ge 2$. Assume additionally that $2n=rs$ where $r$ is an integer.\footnote{This last assumption is automatically satisfied (with $r=n$) for $s=2$, i.e. when $X$ is simply connected.}
Then
\bi
\item[(a)] $\TC(X)\leq r+1$;
\item[(b)] $\TC(X)= r+1$ if and only if there exists a finitely generated abelian group $B$ and a cohomology class $v\in H^s(X;B)$
such that the $n$-fold cup-product of the corresponding zero-divisors (\ref{bar}) is nonzero
$${\bar v}^r=\bar v\cup \dots\cup \bar v \ne 0 \, \in H^{2n}(X\times X; B^r).$$
Here $B^r$ denotes the $r$-fold tensor power
    $B\otimes \dots \otimes B$;
\item[(c)] If $H_\ast(X;\Z)$ is square-free, then $\TC(X)=r+1$ if and
  only if there exists a field $\kk$
 and cohomology classes $v_1,\ldots,v_r\in H^s(X;\kk)$ such that
 $$\bar v_1\cup \dots \cup \bar v_r \ne 0 \in H^{2n}(X\times X; \kk);$$
\item[(d)] If $H_s(X;\Z)$ is free abelian, then $\TC(X)=r+1$ if and
  only if there exist classes $v_1,\ldots ,v_r\in H^s(X;\Z)$ such that
   $$\bar v_1\cup \dots \cup \bar v_r \ne 0 \in H^{2n}(X\times X; \Z).$$
\ei
\end{theorem}

\section{Proofs of Theorems \ref{thm2} and \ref{thm3}}\label{prfs}

\begin{proof}[Proof of Theorem \ref{thm2}] The first statement follows from \cite{Fa03}, Theorem 4.
If there exists a local system coefficient system $A$ and a zero-divisor $\xi\in H^1(X\times X; A)$ such that $\xi^{2n}\not=0$ then $\tc(X) \ge 2n+1$, by
Corollary 4.40 of \cite{Finv}. The remaining part of Theorem \ref{thm2} was proven in \cite{CF}, Theorem 7. More precisely, let $G=\pi_1(X, x_0)$ denote the fundamental group of $X$ and let $I\subset \Z[G]$ denote the augmentation ideal. $I$ can be viewed as a left $\Z[G\times G]$-module via the action
$$(g, h)\cdot \sum n_ig_i= \sum gg_ih^{-1},$$ where $g,h, \in G$ and $\sum n_ig_i\in I$; this defines a local system with stem $I$ on $X\times X$, see \cite{Wh}, chapter 6. A crossed homomorphism $f: G\times G\to I$ given by the formula $$f(g, h)=gh^{-1} -1, \quad g, h \in G$$ determines a cohomology class $\vv\in H^1(X\times X;I)$. This class is a zero-divisor and has the property that $\vv^{2n}\not=0$ assuming that $\tc(X)=2n+1$ according to Theorem 7 from \cite{CF}.
\end{proof}

\begin{proof}[Proof of Theorem \ref{thm3}]
Statement (a) follows directly from Theorem 5.2 of \cite{Fa04} which states that
\begin{equation}
\TC(X)<\frac{2n+1}{s}+1.
\end{equation}
for any $(s-1)$-connected CW-complex  $X$ of dimension $n$.

(b) One part of statement (b) follows from Corollary 4.40 of \cite{Finv}; indeed if $\bar v^{r}\not=0$ then $\tc(X)\ge r+1$ since each $\bar v$ is a zero-divisor.

The proof of the remaining part of statement (b) is derived from obstruction theory and results of
A.\ S.\ Schwarz \cite{Sch} centered around the notion of genus of a fibration.
We assume that $X$ is $(s-1)$-connected, $s\ge 2$, and $n$-dimensional and $2n=rs$ where $r$ is an integer. The case $n=1$ is trivial, therefore we will assume that $n\ge 2$. We want to show that $\tc(X)=r+1$ implies that $\bar v^r\not=0 \in H^{2n}(X\times X; B^r)$ for some class $v\in H^s(X;B)$.

Recall that $\tc(X)$ is defined as the genus of the path fibration (\ref{fibration}) and according to Theorem 3 from \cite{Sch} one has $\tc(X)\le r$ if and only if the $r$-fold fiberwise join
\begin{eqnarray}\label{fibrationr}
\pi_r: P_rX\to X\times X
\end{eqnarray}
of the original fibration $\pi: PX\to X\times X$ admits a continuous section. Hence our assumption $\tc(X)=r+1$ implies that
$\pi_r$ has no continuous sections.
The fibre $F_r$ of (\ref{fibrationr}) is the $r$-fold join
\begin{eqnarray}\label{join}F_r \, = \, \Omega X\ast \Omega X\ast \dots\ast \Omega X\end{eqnarray}
where $\Omega X$ denotes the space of loops in $X$ starting and ending at the base point $x_0\in X$.
Note that $\Omega X$ is $(s-2)$-connected and therefore the fibre $F_r$ is $(2n-2)$-connected since\footnote{One knows that the join a $p$-connected complex and a $q$-connected complex is $(p+q+2)$-connected.}
$r(s-2)+2(r-1)=  2n-2.$

The primary obstruction to the existence of a section of (\ref{fibrationr}) is an element $\theta_r \in H^{2n}(X\times X; \pi_{2n-1}(F_r))$.
It is in fact the only obstruction since the higher obstructions land in zero groups. Thus we obtain that $\theta_r\not=0$.
By the Hurewicz theorem $$\pi_{2n-1}(F_r) = H_{2n-1}(F_r)= B\otimes B\otimes \dots \otimes B = B^r$$
where $B$ denotes the abelian group $H_{s-1}(\Omega X) = H_s(X)$.
Here we have used the K\"unneth theorem for joins, see for instance \cite{Sch}, chapter 1, \S 5.
By Theorem 1 from \cite{Sch} the obstruction $\theta_r$ equals the $r$-fold cup-product
%$$\theta_r=\bar v \cup \dots \cup \bar v = \bar v^r$$
$$\theta_r = \theta\cup \dots \cup \theta = \theta^r$$
where $\theta\in H^s(X\times X; B)$ is the primary obstruction to the existence of a section of $\pi: PX \to X\times X$.
Writing $\theta = v\times 1 + 1\times w$ and observing that $\theta|\Delta_X=0$ (since there is a continuous section of (\ref{fibration}) over the diagonal
$\Delta_X\subset X\times X$) shows that $v+w=0$ and therefore $\theta = v\times 1 - 1\times v=\bar v$. Hence we have found a cohomology class $v\in H^s(X;B)$ with
 $\bar v^r\not=0$.

(c) In one direction the statement of (c) follows from the upper bound (a) and \cite{Fa03}, Thm. 7, i.e. the existence of classes $v_1, \dots, v_r \in H^s(X;\kk)$ with
$\bar v_1\cup \dots\cup \bar v_r \not=0$ combined with (a) gives $\tc(X)=r+1$.
Suppose now that $H_\ast(X)$ is square free. Write $B=H_s(X)$ as a direct sum $$B=\oplus_{i\in I} B_i$$ where each $B_i$ is either $\Z$ or a cyclic group of prime order $\Z_p$ and $I$ is an index set. The $r$-fold tensor power $B^r=B\otimes \dots \otimes B$ is a direct sum
$$B^r= \bigoplus_{(i_1, \dots, i_r)\in I^r}B_{i_1}\otimes B_{i_2} \otimes \dots \otimes B_{i_r}$$
and each tensor product $B_{i_1}\otimes B_{i_2} \otimes \dots \otimes B_{i_r}$ is either $\Z$, $\Z_p$ or trivial. As we know from the proof of (b) there is a class $v\in H^s(X;B)$ such that $\bar v^r\not=0\in H^{2n}(X\times X; B^r)$. For any index $i\in I$ denote by $v_i$ the image of $v$ under the coefficient projection $B\to B_i$.
Since $\bar v^r\not=0$ there exists a sequence $(i_1, \dots, i_r)\in I^r$ such that the product
$$\bar v_{i_1} \cup \dots \cup \bar v_{i_r}\in H^{2n}(X\times X; B_{i_1}\otimes B_{i_2} \otimes \dots \otimes B_{i_r}).$$
is nonzero. If the product $B_{i_1}\otimes B_{i_2} \otimes \dots \otimes B_{i_r}$ is $\Z_p$ then each $B_{i_j}$ is either $\Z$ or $\Z_p$ and
taking $\kk=\Z_p$ and reducing all these classes $v_{i_k}$ mod $p$ we obtain that (c) is satisfied. In the case
when the product $B_{i_1}\otimes B_{i_2} \otimes \dots \otimes B_{i_r}$ is infinite cyclic each of the groups $B_{i_k}$ is $\Z$ and the class
\begin{eqnarray}\label{prod}
\bar v_{i_1}\cup \dots \cup \bar v_{i_r}\not= 0 \in H^{2n}(X\times X; \Z)
\end{eqnarray}
is integral and nonzero.

Since the group $H^{2n}(X\times X; \Z)$ is square-free the cup-product (\ref{prod}) is indivisible by some prime $p$.
Indeed, the group $H^{2n}(X\times X; \Z)$ is direct sum of cyclic groups of prime order and infinite cyclic groups and the product (\ref{prod}) has a nontrivial component in at least one of these groups. A nonzero element of $\Z$ is divisible by finitely many primes and a nonzero element of $\Z_p$ is divisible by all primes except $p$.

Therefore, as follows from the exact sequence
 $$\dots \to H^{2n}(X\times X;\Z) \stackrel{p}\to H^{2n}(X\times X;\Z) \to H^{2n}(X\times X;\Z_p)\to \dots, $$
 the mod $p$ reduction of the product (\ref{prod}) is nonzero.
Now, taking $\kk=\Z_p$ and reducing the classes $v_{i_k}$ mod $p$ gives a sequence of classes $w_{j_k}\in H^s(X;\kk)$ such that $\prod \bar w_{j_k}\not =0$ where
$k=1, \dots, r$.

(d) The proof of statement (d) of Theorem \ref{thm3} is similar to that of (c), with the simplification that all the groups $B_i$ are in this case infinite cyclic.
\end{proof}

\section{Proof of Theorem 1}\label{prf}

The cases $m=2$ and $m\geq 3$ odd of Theorem 1 were dealt with by Farber
and Yuzvinsky in \cite{FY04}. Their arguments also show that if $m\geq
4$ is
even, then $\TC(F(\R^m,n))$ equals
either $2n-1$ or $2n-2$. Hence to prove Theorem 1 it suffices to show
that $\TC(F(\R^m,n))\neq 2n-1$ when $m\geq 4$ is even.

Fix $n\geq 2$. For any $m\geq 2$ the space $F(\R^m,n)$ is $(m-2)$-connected, since it is the complement
of an arrangement of codimension $m$ subspaces of $\R^{mn}$. Its
integral cohomology ring is shown in \cite{FH} to be graded-commutative
algebra over $\Z$ on generators $$e_{ij}\in H^{m-1}(F(\R^m,n)), \quad 1\leq i<j\leq
n,$$ subject to the relations $$e_{ij}^2=0, \quad e_{ij}e_{ik}=(e_{ij}-e_{ik})e_{jk}$$ for
any triple $1\leq i<j<k\leq n$. In particular, $H^*(F(\R^m,n))$ is nonzero only in dimensions $i(m-1)$ where $i=0, 1, \dots, (n-1)$.
Applying the result of Eilenberg and Ganea \cite{EG} we obtain that for $m\ge 3$ the space $F(\R^m, n)$ is homotopy equivalent to a finite complex of dimension $\le (m-1)(n-1)$. Now we may apply statement (d) of Theorem \ref{thm3}, which gives, firstly, that $\tc(F(\R^m, n))\le 2n-1$ and, secondly,
$\tc(F(\R^m, n))= 2n-1$ if and only if there exist cohomology classes $$v_1, \dots,v_{2(n-1)}\in H^{m-1}(F(\R^m, n))$$ such that the product of the corresponding zero-divisors $$\bar v_1\cup \bar v_2\cup \dots \cup \bar v_{2(n-1)}$$  is nonzero; recall that the notation $\bar v$ is introduced in (\ref{bar}).
We show below that such classes $v_1, \dots, v_{2(n-1)}$ do not exists if $m\ge 4$ is even.

We recall the result of \cite{FY04} stating that $\tc(F(\C, n))=2n-2$. It is shown in the proof of Theorem 6 in \cite{FY04},
that
$F(\C, n)$ is homotopy equivalent to the product
$X\times S^1$ where $X$ is a finite polyhedron of dimension $\le n-2$. This argument uses the algebraic structure of
$\C=\R^2$ and does not generalize to $F(\R^m,n)$ with $m>2$.
Using the product inequality (Theorem 11 in \cite{Fa03}) one obtains
\begin{eqnarray*}\tc(F(\C, n)) &\le& \tc(X) +\tc(S^1)-1\\
&\le& (2(n-2)+1)+2-1=2n-2.\end{eqnarray*}
Hence there exist no $2(n-1)$ cohomology classes $v_1, \dots, v_{2(n-1)}\in H^1(F(\C, n))$ such that the product of the zero-divisors
$\bar v_1 \cup \dots \cup \bar v_{2(n-1)}$ is nonzero,
as this would contradict Theorem 7 from \cite{Fa03}.

Now we observe that for any even $m\geq 2$ there is an algebra isomorphism
\begin{eqnarray}
\phi: H^*(F(\C;n)) \to H^{*(m-1)}(F(\R^m,n))
\end{eqnarray}
mapping classes of degree $i$ to classes of degree $(m-1)i$ where $i=0, 1, \dots, n-1$, see \cite{FH}. Thus we conclude that there exist no cohomology classes
$w_1, \dots, w_{2(n-1)}\in H^{m-1}(F(\R^m, n))$ such that the product of the corresponding zero-divisors
$\bar w_1 \cup \dots \cup \bar w_{2(n-1)}$ is nonzero. Theorem \ref{thm3} (statement (d)) gives now that $\tc(F(\R^m,n))\le 2(n-1)$.

On the other hand, it is proven in \cite{FY04} that one may find $2n-3$ cohomology classes $v_1, \dots, v_{2n-3}\in H^1(F(\C,n))$ such that the cup-product
$\bar v_1 \cup \dots\cup \bar v_{2n-3}$ is nonzero. Hence, repeating the above argument we see that for $m$ even there exists classes
$w_1, \dots, w_{2n-3}\in H^{m-1}(F(\R^m,n))$ (where $w_i=\phi(v_i)$) with nonzero product $\bar w_1 \cup \dots \cup \bar w_{2n-3}$; this gives the opposite inequality
$\tc(F(\R^m, n)) \ge 2n-2$.

Hence, $\tc(F(\R^m, n))=2n-2$ as stated.
\qed

\end{document}